\documentclass{article}
\usepackage{graphicx}
\usepackage{amsmath}
\usepackage{amsfonts}
\usepackage{amssymb}

\newtheorem{theorem}{Theorem}

\newtheorem{claim}[theorem]{Claim}

\newtheorem{corollary}[theorem]{Corollary}

\newtheorem{definition}[theorem]{Definition}

\newtheorem{lemma}[theorem]{Lemma}

\newtheorem{proposition}[theorem]{Proposition}

\newenvironment{proof}[1][Proof]{\noindent\textbf{#1.} }{\ \rule{0.5em}{0.5em}}

\newcommand{\N}{\mathbb{N}}

\numberwithin{equation}{section}

\begin{document}

\title{Computing geometric Lorenz attractors with arbitrary precision}

\author{Daniel Gra\c{c}a, Cristobal  Rojas and Ning Zhong}

\maketitle

%
%
%


\begin{abstract}
The Lorenz attractor was introduced in 1963 by E. N. Lorenz as one of the
first examples of \emph{strange attractors}. However Lorenz' research was mainly based on
(non-rigourous) numerical simulations and, until recently, the proof of the existence of the Lorenz attractor remained elusive.
To address that problem some authors introduced geometric Lorenz models and proved that geometric Lorenz models
have a strange attractor. In 2002 it was shown that the original Lorenz model behaves like a geometric Lorenz model and thus has a strange attractor.

In this paper we show that geometric Lorenz attractors are computable, as well as their physical measures.

\end{abstract}

\section{Introduction}

The system of equations
\begin{equation}
\left\{
\begin{array}
[c]{c}%
x^{\prime}=\sigma(y-x)\\
y^{\prime}=\rho x-y-xz\\
z^{\prime}=xy-\beta z
\end{array}
\right.  \label{ODE:Lorenz}%
\end{equation}
is called the Lorenz system, where $\sigma,\beta,$ and $\rho$ are parameters.
This system was first studied by E. N. Lorenz in 1963 \cite{Lor63} as a
simplified model of atmosphere convection in an attempt to understand the
unpredictable behavior of the weather. Lorenz's original numerical
simulations, where the parameters were given by $\sigma=10$, $\beta=8/3$, and
$\rho=28$, suggested that for any typical initial condition, the system\ would eventually tend to a same limit set with a
rather complicated structure --  the \emph{Lorenz (strange) attractor}. Moreover, the dynamics on this attractor seemed
to magnify small errors very rapidly, rendering impractical to numerically simulate an individual trajectory for an extended period of time.

The Lorenz system became a landmark in  the modern paradigm of
the numerical study of chaos: instead of studying trajectories individually, one should
study the limit set of a typical orbit, both as a spatial object and as a statistical
distribution \cite{Palis}.   However,
proving the existence of the Lorenz attractor in a rigorous fashion turned out
to be no easy task; indeed, the problem was listed in 1998 by Smale as one of
the eighteen unsolved problems he suggested for the 21st century \cite{Sma98}.

In 1979, based on the behavior observed in the numerical simulations of
(\ref{ODE:Lorenz}), Afraimovich, Bykov, and Shil'nikov \cite{ABS77}, and
Guckenheimer and Williams \cite{GW79} originated the study of flows
satisfying a certain list of geometric properties intended to capture the observed numerically
simulated behavior. In particular, they proved that any such flow must
contain a strange attractor, which supports a unique invariant probability
distribution that describes the limiting statistical behavior of almost any initial condition.
These examples came to be known as geometric
Lorenz models, and the strange attractor contained in a geometric Lorenz flow
is called the geometric Lorenz attractor.


Using a combination of normal form theory and rigorous numerics, Tucker
\cite{Tuc02} provided, in 2002, a formal proof on the existence of the Lorenz
attractor by showing that the geometric Lorenz models do indeed correspond to
the Lorenz system (\ref{ODE:Lorenz}) for certain parameters. Since a geometric
Lorenz model supports a strange attractor, so does the Lorenz system
(\ref{ODE:Lorenz}).

In this note, we examine computability of geometric Lorenz attractors and
their physical measures.  By definition, a computable set in the plane can be
visualized on a computer screen with an arbitrarily high
magnification, and integrals with respect to a computable probability measure can be
generated by a computer with arbitrary precision.  Our main result is the following:

\medskip
\noindent\textbf{Main Theorem}. \emph{For any geometric Lorenz flow, if the data defining the flow are
computable, then its attractor is a computable  subset of
$\mathbb{R}^{3}$. Moreover, the physical measure supported on this attractor is a computable probability
measure.}

\medskip

We note that, although computer generated images of the ``butterfly shaped'' Lorenz attractor
abound in the internet, these images are not rigorous computations. In particular,
their existence does not necessarily mean that the attractor is actually computable.
In fact, an equally famous collection of invariant sets, namely \emph{Julia sets},
whose computer images are also abundant, was shown to contain non computable members \cite{BY06}.

In order to make our results accessible
to a wide audience, we have made an effort to work directly from the definitions, making the proofs as self-contained as possible.

\section{Preliminaries}

\label{Sec:prelims}

\subsection{Computable analysis}

\label{Subsec:comptanalysis}

Roughly speaking, an object is computable if it can be approximated by
computer-generated approximations with an arbitrarily high precision.
Formalizing this idea to carry out computations on infinite objects such as
real numbers, we encode those objects as infinite sequences of rational
numbers (or equivalently, sequences of any finite or countable set $\Sigma$ of
symbols), using representations (see \cite{Wei00} for a complete development).
A represented space is a pair $(X; \delta)$ where $X$ is a set,
$\mbox{dom}(\delta)\subseteq\Sigma^{\mathbb{N}}$, and $\delta: \subseteq
\Sigma^{\mathbb{N}}\to X$ is an onto map (``$\subseteq\Sigma^{\mathbb{N}}$" is
used to indicate that the domain of $\delta$ may be a subset of $\Sigma
^{\mathbb{N}}$). Every $q\in\mbox{dom}(\delta)$ such that $\delta(q)=x$ is
called a $\delta$-name of $x$ (or a name of $x$ when $\delta$ is clear from
context). Naturally, an element $x\in X$ is computable if it has a computable
name in $\Sigma^{\mathbb{N}}$ (the notion of computability on $\Sigma
^{\mathbb{N}}$ is well established). In this note, we use the following
particular representations for points in $\mathbb{R}^{n}$; for closed subsets
of $\mathbb{R}^{n}$; and for continuous functions defined on $I_{1}\times
I_{2}\times\cdots\times I_{n}\subset\mathbb{R}^{n}$, where $I_{j}$'s are intervals:

\begin{itemize}
\item[(1)] For a point $x\in\mathbb{R}^{n}$, a name of $x$ is a sequence $\{
r_{k}\}$ of points with rational coordinates satisfying $|x-r_{k}|<2^{-k}$.
Thus $x$ is computable if there is a Turing machine (or a computer program or
an algorithm) that outputs a rational $n$-tuple $r_{k}$ on input $k$ such that
$|r_{k}-x|<2^{-k}$; for a sequence $\{ x_{j}\}$, $x_{j}\in\mathbb{R}^{n}$, a
name of $\{x_{j}\}$ is a double sequence $\{ r_{j,k}\}$ of points with
rational coordinates satisfying $|x_{j}-r_{j, k}|<2^{-k}$.

\item[(2)] For a closed subset $A$ of $\mathbb{R}^{n}$, a name of $A$ consists
of a pair of an inner-name and an outer-name; an inner-name is a
sequence dense in $A$ and an outer-name is a sequence of balls
$B(a_{n}, r_{n})=\{ x\in\mathbb{R}^{n} \, : \, d(a_{n}, x)<r_{n}\}$,
$a_{n}\in\mathbb{Q}^{n}$ and
$r_{n}\in\mathbb{Q}$, exhausting the complement of $A$, i.e., $\mathbb{R}%
^{n}\setminus A=\bigcup_{n=1}^{\infty}B(a_{n}, r_{n})$. $A$ is said
to be r.e. closed if the sequence (dense in $A$) is computable;
co-r.e. closed if the sequences $\{ a_{n}\}$ an $\{ r_{n}\}$ are
computable; and computable if it is r.e. and co-r.e.. For a compact
set $K$, a name of $K$ consists of a name of $K$ as a closed set and
a rational number $r$ such that $K\subseteq B(0, r)$. By the
definition, a planar computable closed set can be visualized on a
computer screen with an arbitrarily high magnification.

\item[(3)] For every continuous function $f$ defined on $I_{1}\times
I_{2}\times\cdots\times I_{n}\subseteq\mathbb{R}^{n}$, where $I_{j}$ is an
interval with endpoints $a_{j}$ and $b_{j}$, a name of $f$ is a double
sequence $\{ P_{k, l}\}$ of polynomials with rational coefficients satisfying
$d_{k}(P_{k, l}, f)<2^{-l}$, where $d_{k}(g, f)=\max\{|g(x)-f(x)|\, : \,
a_{j}+2^{-k}\leq x_{j}\leq b_{j}-2^{-k}, 1\leq j\leq n\}$ ($d_{k}(g, f)=0$ if
$[a_{j}+2^{-k}, b_{j}-2^{-k}]=\emptyset$). Thus, $f$ is computable if there is
an (oracle) Turing machine that outputs $P_{k, l}$ (more precisely
coefficients of $P_{k, l}$) on input $k, l$ satisfying $d_{k}(P_{k, l},
f)<2^{-l}$.

\item[(4)] For every $C^m$ function $f$ defined on $E=I_{1}\times
I_{2}\times\cdots\times I_{n}\subseteq\mathbb{R}^{n}$, where $I_{j}$ is an
interval with endpoints $a_{j}$ and $b_{j}$, a ($C^{m}$) name of $f$ is a double sequence
$\{P_{k,l}\}$ of polynomials
with rational coefficients satisfying
\[
d^m _{k}(P_{k, l}, f)<2^{-l},
\]
where
\[
d^m_{k}(g, f)=\max_{0\leq i\leq m}\max\{|D^{i}g(x)-D^{i}f(x)|\, : \,
a_{j}+2^{-k}\leq x_{j}\leq b_{j}-2^{-k}\}
\]
($d^m_{k}(g, f)=0$ if
$[a_{j}+2^{-k}, b_{j}-2^{-k}]=\emptyset$).
We observe that a $C^m$ name of $f$ contains information on both $f$ and
$Df, D^{2}f,\ldots,D^{m}f$,  in the sense that $(P_{1}, P_{2}, \ldots)$ is a $\rho$-name of $f$ while
$(D^{i}P_{1}, D^{i}P_{2}, \ldots)$ is a $\rho$-name of $D^{i}f$. See \cite{ZW03} for
further details.
\end{itemize}

The notion of computable maps between represented spaces now arises naturally.
A map $\Phi: (X;\delta_{X})\to(Y;\delta_{Y})$ between two represented spaces
is computable if there is a computable map $\phi:\subseteq\Sigma^{\mathbb{N}%
}\to\Sigma^{\mathbb{N}}$ such that $\Phi\circ\delta_{X}=\delta_{Y}\circ\phi$.
Informally speaking, this means that there is a computer program that outputs
a name of $\Phi(x)$ when given a name of $x$ as input \cite{BHW08}.

\subsection{Geometric Lorenz models}

\label{Subsec:LorenzGeometric}

We briefly describe a geometric Lorenz model taken from \cite{GH83} (see
section 5.7 \cite{GH83} for more details).

For the parameter values $\sigma=10$, $\beta=8/3$, and $\rho=28$, the Lorenz
system (\ref{ODE:Lorenz}) has three equilibrium points: the origin, $q_{-}$,
and $q_{+}$;
both $q_{-}$ and $q_{+}$ lie on the plane $z=\rho-1=27$. The numerical
simulations of (\ref{ODE:Lorenz}) for these parameter values show that the
Lorenz flow rotates around the equilibria $q_{\pm}$ and intersects the plane
$z=27$ infinitely many times, thus indicating that there is a return map with
the cross section $z=27$. Geometric Lorenz models are constructed based upon
the behavior of this numerically observed return map.

It is proved (cf. \cite{GH83}, \cite{HSD04} and references therein)
that a flow satisfying the following properties exists (see Figure
\ref{Fig:GeometricLorenz}): The flow has three equilibrium points:
the origin of $\mathbb{R}^{3}$ and $Q_{\pm}$; for the origin, its
stable manifold is the $yz$-plane while its unstable manifold
intersects the plane $z=27$ from above at two points, say
$\rho^{+}=(r^{-},t^{-})$ and $\rho^{-}=(r^{+},t^{+})$; for $Q_{-}$
and $Q_{+}$, they lie in the plane $z=27$ and have integer
coordinates $(-m, -n, 27)$ and $(m, n, 27)$, their stable lines are
parallel to the $y$-axis, and the other two eigenvalues at $Q_{\pm}$
are assumed to be complex with positive real part, as is the Lorenz
system. Let $\Sigma$ be a rectangle contained in the plane $z=27$
such that $\rho^{\pm}$ is contained in $\Sigma$, the two opposite
sides of $\Sigma$ parallel to the $y$-axis pass through the
equilibrium points $Q_{-}$ and $Q_{+}$, and these two sides form
portions of the stable lines at $Q_{-}$ and $Q_{+}$. Let $D$ be the
intersection of the $yz$-plane and $\Sigma$. The flow has the
following features: $\Sigma$ is a cross section for the flow; all
trajectories go downwards through $\Sigma$; all trajectories
originating in $\Sigma$ and not entering $D$ spiral around $Q_{-}$
or $Q_{+}$ and return to $\Sigma$ as time moves forward; all
trajectories beginning at points in $D$ tend to the origin as time
moves forward and never return to $\Sigma$; and there is a
Poincar\'{e} return map $F: \Sigma_{-}\bigcup\Sigma_{+}\to\Sigma$,
where $\Sigma_{-}=\{(x,y)\in \Sigma|x<0\}$ and
$\Sigma_{+}=\{(x,y)\in\Sigma|x>0\}$.

Let $V=\{(x,y)|r^{-}\leq x\leq r^{+}, \, -27\leq y\leq27\}$ (the
number 27 is arbitrarily chosen; other positive numbers can be used
as well). The Lorenz flow has also the property that all points in
the interior of $\Sigma \setminus D$ have a trajectory which will
eventually reach $V$ and
$F(V\setminus D)\subseteq V$ (see Figure
\ref{Fig:LorenzPlaneReduced}). Thus we can restrict the analysis of
the flow to $V$. The Poincar\'{e} return map $F$ has the following
properties on $V$: 

\begin{itemize}
\item[(F-1)] The set $\mathcal{F}$, $\mathcal{F}=\{ \mbox{$x$=constant}\}$,
is invariant under the action of $F$. In other words, the $x$-coordinate of the
image $F(x_0,y_0)$ depends only on $x_0$.


\item[(F-2)] There are functions $f$ and $g$ such that $F$ can be written as%
\[
F(x,y)=(f(x),g(x,y))\text{ \ \ for }x\neq0
\]
and $F(-x,-y)=-F(x,y)$.

\item[(F-3)] $f^{\prime}(x)>\sqrt{2}$ for $x\neq0$ and $f^{\prime
}(x)\rightarrow\infty$ as $x\rightarrow0$; $0<f(r^{+})<r^{+}$ and
$r^{-}<f(r^{-})<0$ (recall that the unstable manifold of the origin first
intersects $V$ from above at points $\rho^{+}$ and $\rho^{-}$).

\item[(F-4)] $0<\partial g/\partial y\leq c<1/\sqrt{2}$ and $0<\partial
g/\partial x\leq c$ for $x\neq0$ and $\partial g/\partial
y\rightarrow0$ as $x\rightarrow0$  (see \cite[Section 14.4]{HSD04}).
Without loss of generality, $c$ can be assumed to be a rational
number and $\partial g/\partial y\rightarrow0$ to be monotonic as
$x\to 0$. \\
\end{itemize}

A consequence of (F-2)-(F-4) is that (see \cite[Section
14.4]{HSD04}): \\

\begin{itemize}
\item[(F-5)] $\lim_{x\rightarrow0^{-}}F(x,y)=(r^{+},t^{+})$ and $\lim
_{x\rightarrow0^{+}}F(x,y)=(r^{-},t^{-})$, where $\rho^{-}=(r^{+},t^{+})$ and
$\rho^{+}=(r^{-},t^{-})$. The symmetry property (F-2) implies that
$r^{-}<0<r^{+}$ and $r^{-}=-r^{+}$. \newline
\end{itemize}

\begin{figure}[ptb]
\begin{center}
\includegraphics[width=0.6\textwidth]{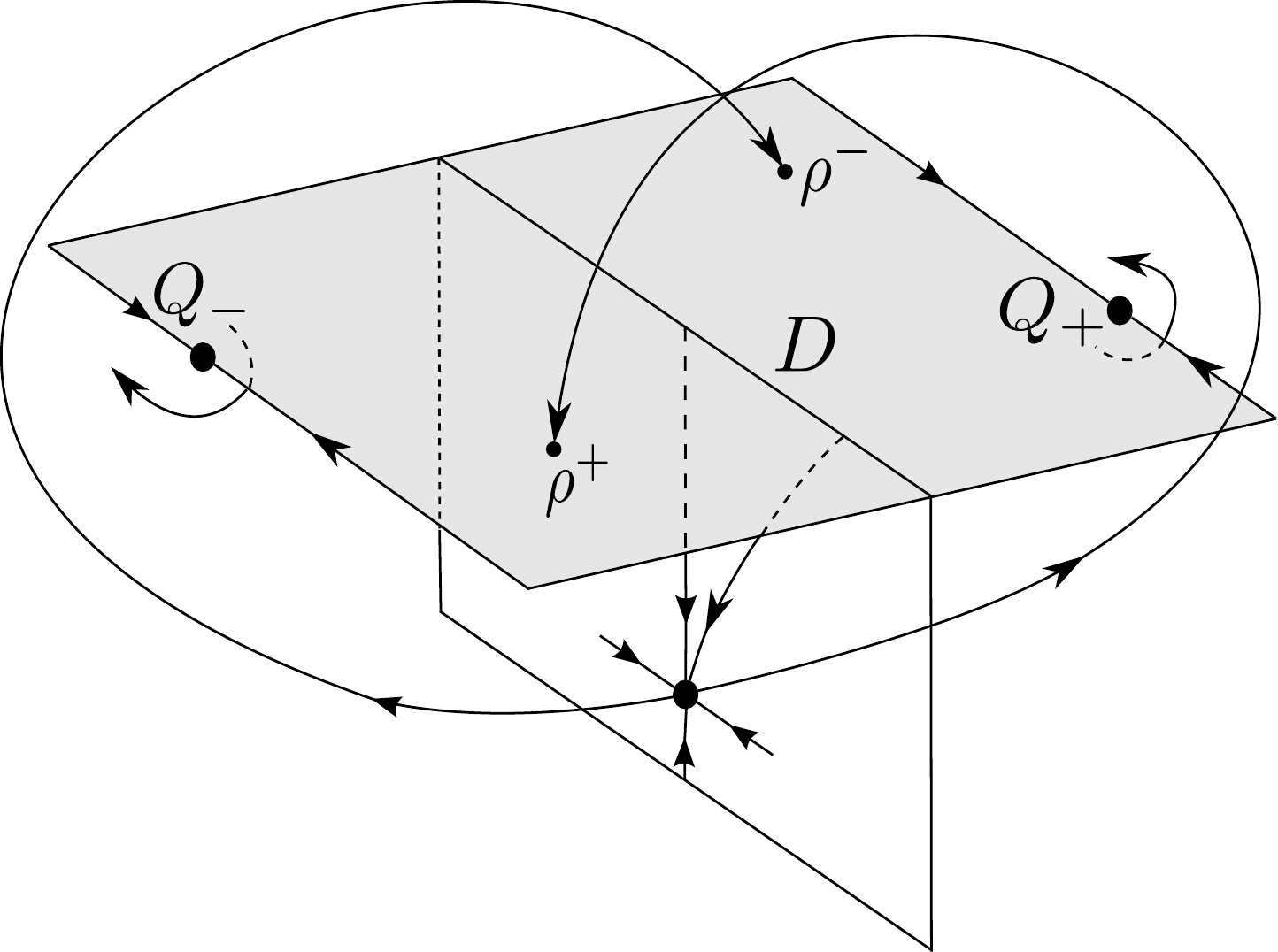}
\end{center}
\caption{The geometric model for the Lorenz system.}%
\label{Fig:GeometricLorenz}%
\end{figure}

The image of $\Sigma$ by $F$ is depicted in Figure \ref{Fig:GeometricLorenzPlane}.
\newline

\begin{figure}[ptb]
\begin{center}
\includegraphics[width=0.8\textwidth]{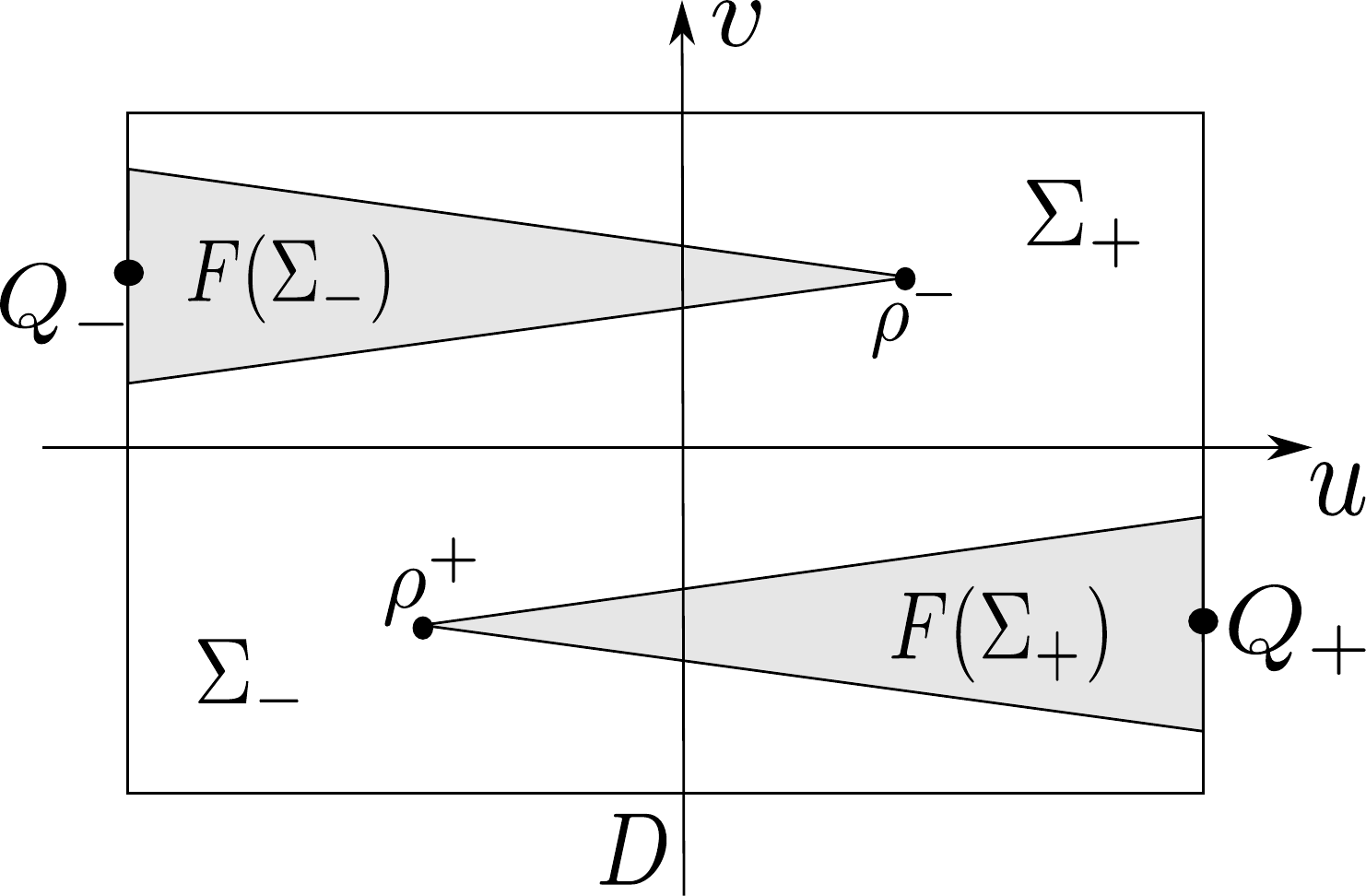}
\end{center}
\caption{The Poincar\'{e} map $F$ on the cross-section $\Sigma$.}%
\label{Fig:GeometricLorenzPlane}%
\end{figure}

Figure \ref{Fig:LorenzSolid} shows a picture of the flow, where $\Sigma$ is the
upper surface of the solid and the flow is tangent to the curved surfaces of
the solid and to the bottom segment. On the front and back surfaces, the flow
is into the solid while the trajectories emerge from the vertical ends. These
emergent trajectories are continued around so that $F$ describes the return
map for $V$. This flow, denoted by $\phi_{t}$, $t\in\mathbb{R}$ and acting on
$M=\{\phi_{t}(x,y,z) : (x,y,z) \in V \times \{27\},  t\in\mathbb{R}^{+}\}$ is called a
geometric Lorenz flow.

\begin{figure}[ptb]
\begin{center}
\includegraphics[width=0.8\textwidth]{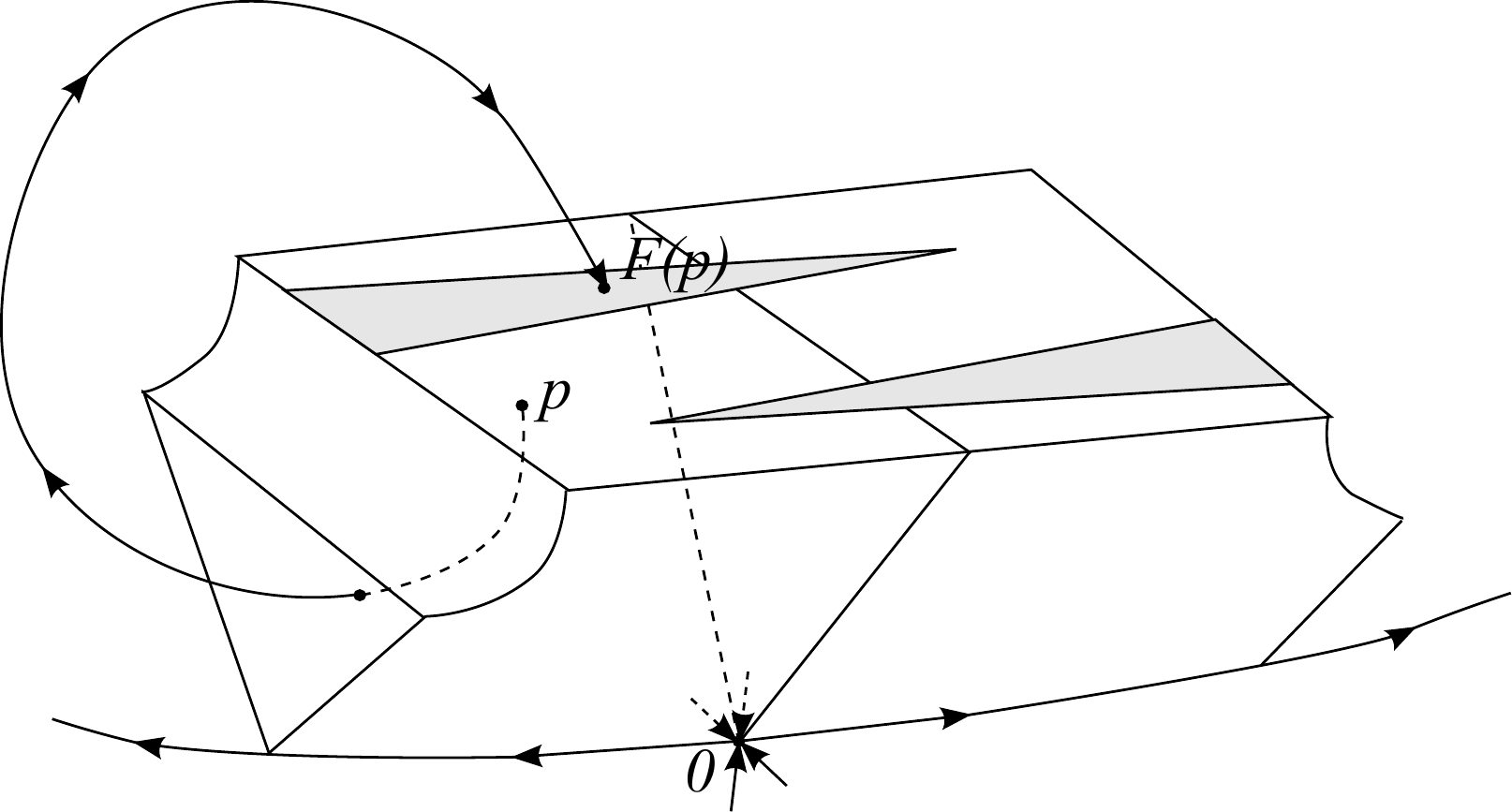}
\end{center}
\caption{A three-dimensional representation of the geometric Lorenz flow.}%
\label{Fig:LorenzSolid}%
\end{figure}


It is shown in \cite{GH83} that%
\[
A=%
{\displaystyle\bigcap\limits_{n\geq0}}
\overline{F^{n}(V\setminus D)}
\]
is the intersection of the attractor for the geometric Lorenz flow
with $V$
and that%
\[
\mathcal{A}=\left(
{\displaystyle\bigcup\limits_{t\in\mathbb{R}}}
\phi_{t}(A)\right)  \cup\{(0,0,0)\}
\]
is an attractor for the geometric Lorenz flow $\phi_{t}$; this
attractor is a Lorenz-like strange attractor. Note that $F$ is
defined on $V\setminus D$. Thus $F^2(V\setminus D)$ is understood as
of $F(F(V\setminus D)\setminus D)$ and, inductively,
$F^{n+1}(V\setminus D)=F(F^n(V\setminus D)\setminus D)$.

We mention in passing that the geometric Lorenz model is not unique; in fact,
any flow which satisfies the geometric conditions listed above contains a
Lorenz-like strange attractor, thus it is a geometric Lorenz flow. As usual,
one might also need to use some reparametrization of the model to ensure that
it behaves like described in this section (it has a fixed point on the origin, etc.).
All computability results stated in this paper are relative to that (eventual) reparametrization.

\begin{figure}[ptb]
\begin{center}
\includegraphics[width=0.6\textwidth]{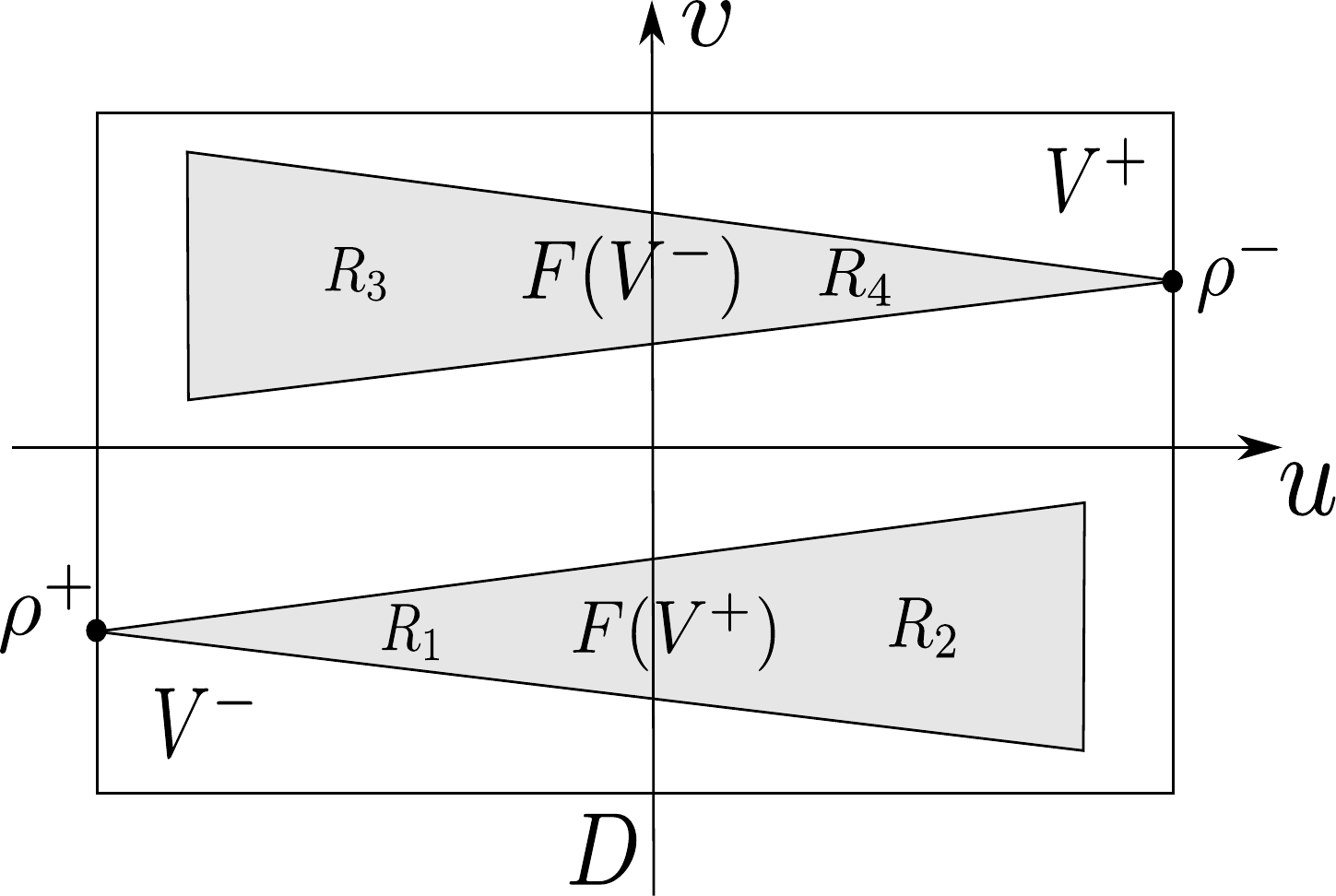}
\end{center}
\caption{The Poincar\'{e} map on the cross-section $V$.}%
\label{Fig:LorenzPlaneReduced}%
\end{figure}

\section{Computability of geometric Lorenz attractors}

\label{Sec:LorenzGeometricComputable}

In this section, we show that the strange attractor $\mathcal{A}$ contained in
a geometric Lorenz flow is uniformly computable from the data defining the
flow. Thus, if the data defining the flow is computable, then so is
$\mathcal{A}$; by definition this means that $\mathcal{A}$ can be visualized
on a computer screen with an arbitrary high magnification.

We begin by studying computability of the set $A$, for $\mathcal{A}$
consists of the trajectories passing through $A$. We start by
showing that $A$ is uniformly computable from $F$ and $\rho^{\pm}$.



\begin{theorem}
\label{Th:ReturnMapAttractor} The operation $(F, \rho^{\pm})\to A$
is computable.


\end{theorem}

\begin{proof}
We show that $A$ is computable by making use of the
``constructive" definition of $A$, $A=\bigcap_{n\geq0}\overline{F^{n}%
(V\setminus D)}$. Let $A_{n}=\overline{F^{n}(V\setminus D)}$. In
Propositions \ref{independence} and \ref{claim-3} below, we show
that
\begin{enumerate}
\item[i)] the sequence $\{ A_{n}\}$ is computable from $F$ and $\rho^{\pm}$ and
\item[ii)] $\max_{(x, y)\in V}|d_{A_{n+1}%
}(x,y)-d_{A_{n}}(x,y)|\leq 108c^{n}$ (see (F-4) for definition of the
number $c$).
\end{enumerate}

Then it follows from (d) and (e) of Lemma \ref{Lm:Weihrauch} that
$A$ is computable from $F$ and $\rho^{\pm}$ since $A_{n+1}\subseteq
A_{n}$.  Propositions \ref{independence} and \ref{claim-3}, together
with their proofs, are presented below.
\end{proof}

\medskip

We will need the following lemmas, which will be also needed in several later proofs. Recall that
an operation $\mathcal{O}: X \to Y$ is computable if there is a
Turing algorithm that, for any given name of $x$ as input, outputs a
name of $\mathcal{O}(x)$.

\begin{lemma}
\label{Lm:Patching} The operation $(f, b, \beta)\to f_{1}$ is computable,
where $f: (0, b]\to\mathbb{R}$ is a continuous function with the following properties:
(i) it is monotonic on $(0, b]$ and (ii) $f(x)\to\beta$ as
$x\to0^{+}$, and
\[
f_{1}(x)=\left\{
\begin{array}
[c]{ll}%
f(x) & x\in(0, b]\\
\beta & x=0
\end{array}
\right.
\]
Similarly, the operation $(f, a, \beta)\to f_{2}$ is computable, where $f: [a,
0)\to\mathbb{R}$ is a continuous function, with the following properties: (i) it
is monotonic on $[a,0)$ and (ii) $f(x)\to\beta$ as
$x\to0^{-}$, and
\[
f_{2}(x)=\left\{
\begin{array}
[c]{ll}%
f(x) & x\in[a, 0)\\
\beta & x=0
\end{array}
\right.
\]

\end{lemma}

\begin{proof}
The proof is straightforward, thus omitted.
\end{proof}

\begin{lemma}
\label{Lm:Weihrauch} The following results can be found in Chapters
5 \& 6 \cite{Wei00}. Assume that $A$, $B$, and $K$ are subsets of
$\mathbb{R}^n$.

\begin{itemize}
\item[(a)] The operation $(f, K)\to f[K]$ for continuous $f$ and compact $K$
is computable.

\item[(b)] The union $(A, B)\to A\bigcup B$ of compact sets is computable.

\item[(c)] The operation $(f, a, b)\to(f^{-1}, c, d)$ for continuous strictly monotonic
$f$ is computable, where $f^{-1}$ is the inverse of $f$ and $f[a, b]=[c, d]$.

\item[(d)] The operation $((f_{i})_{i\in\mathbb{N}}, K)\to f$, $f(x)=\lim
_{i\to\infty}f_{i}(x)$, is computable, where $\max_{x\in K}|f_{i}%
(x)-f_{j}(x)|\leq C\cdot c^{i}$ for all $j>i$, $K$ is compact, $C$ is a
rational number, and $0<c<1$.

\item[(e)] The operations $K\to d_{K}$ and $d_{K}\to K$ for non-empty compact sets $K$
are computable, where $d_{K}$ is the distance function defined on $K$:
$d_{K}(x)=dist (x, K)$.
\end{itemize}
\end{lemma}

\begin{proposition} \label{independence}
$\{A_{n}\}$ is computable.
\end{proposition}
\begin{proof}
Let us first give an idea on how to compute the sequence $\{
A_{n}\}$ from $F$ and $\rho^{\pm }$. Intuitively one may attempt to
compute $F^{n}(V\setminus D)$ directly and then compute its closure.
However, since $F^{n}(V\setminus D)$ is neither open nor closed, it
cannot be computed from $F^{n}$, although $F^{n}$ is computable from
$F$ for every $n\in\mathbb{N}$. A possible solution is to extend $F$
to be defined on $V$ and then compute $F^{n}(V)$. But this method
also fails to work here -- $F$ is singular along $D\subseteq V$ and therefore it
cannot be extended to a continuous function on $V$. Nevertheless, we
observe that if we break $F$ into two half functions, each of them
 defined on one half of $V\setminus D$,  then we can extend each
half function continuously to be also defined on $D$ (from property
(F-5)); moreover, the extension is computable from the given data.
And so if we can show that the iterations of the two half functions
yield $\overline{F^{n}(V\setminus D)}$ and are computable from $F$
and $\rho^{\pm}$, we have the desired sets $A_{n}$.

Now for the details. Recall that $V=\{(x,y)|r^{-}\leq x\leq r^{+}, \, -27\leq
y\leq27\}$. Let $V^{+}=\{(x,y)\in V\,|\,x\geq0\}$ and $V^{-}=\{(x,y)\in
V\,|\,x\leq0\}$; let $I^{+}=\{x \, | \,0\leq x\leq r^{+}\}$, $I^{-}=\{x \, \,
|r^{-}\leq x\leq0\}$, and $I=[r^{-},r^{+}]$. It is then clear that $V$,
$V^{+}$, $V^{-}$, $I$, $I^{+}$, and $I^{-}$ are all computable from $r^{\pm}$.
Also recall that $F(x, y)=(f(x), g(x, y))$ on $V\setminus D$, $D=\{ (0, y) \,
| \, -27\leq y\leq27\}$.
Define

\begin{align*}
f^{+}:I^{+}\rightarrow I,\quad f^{+}(x)  &  =\left\{
\begin{array}
[c]{ll}%
f(x) & 0<x\leq r^{+}\\
r^{-} & x=0
\end{array}
\right.\\
f^{-}:I^{-}\rightarrow I,\quad f^{-}(x)  &  =\left\{
\begin{array}
[c]{ll}%
f(x) & r^{-}\leq x<0\\
r^{+} & x=0
\end{array}
\right.\\
g^{+}:V^{+}\rightarrow [-27,27],\quad g^{+}(x,y)  &=\left\{
\begin{array}
[c]{ll}%
g(x,y) & 0<x\leq r^{+}\\
t^{-} & x=0
\end{array}
\right.\\
g^{-}:V^{-}\rightarrow [-27,27],\quad g^{-}(x,y)  &=\left\{
\begin{array}
[c]{ll}%
g(x,y) & r^{-} \leq x< 0\\
t^{+} & x=0
\end{array}
\right.
\end{align*}
By (F-3) and (F-5) (that is, $f^{\prime}(x)>\sqrt{2}$ for $x\neq0$, $\lim
_{x\to0^{-}}F(x, y)=(r^{+},t^{+})$ and $\lim_{x\to0^{+}}F(x, y)=(r^{-},t^{-}%
)$), it follows that $f(x) \nearrow r^{+}$ when $x\to 0^-$  and that  $f(x)\searrow r^{-}$ as $x\to 0^+$.
Then it follows from Lemma \ref{Lm:Patching} that
both $f^{+}$ and $f^{-}$ are computable. A similar argument shows that $g^+,g^-$ are computable.

Let $F^{\pm}(x, y)=(f^{\pm}(x), g^{\pm}(x, y))$, $(x, y)\in
V^{\pm}$, where $f^{\pm}$ ($g^{\pm}$) is either $f^{+}$  or $f^{-}$
($g^+$ or $g^-$). Then $F^{\pm}$ is computable from $f^{\pm}$ and
$g^{\pm}$; thus it is computable.


Now, recall that  $A_n=\overline{F^n(V\setminus D)}$ by definition. It
can be shown that
\begin{equation} \label{A-n-description}
A_n=\overline{F(A_{n-1}\setminus D)}, \quad n\geq 1
\end{equation}
by induction as follows: Let $A_0=V$. Then
$A_1=\overline{F(V\setminus D)}=\overline{F(A_0\setminus D)}$.
Assume that $A_n=\overline{F(A_{n-1}\setminus D)}$. We show that
$A_{n+1}=\overline{F(A_n\setminus D)}$. By definition,
\begin{eqnarray*}
A_{n+1} & = & \overline{F^{n+1}(V\setminus D)} \\
& = & \overline{F(F^n(V\setminus D)\setminus D)} \\
& \subseteq &
\overline{F\left(\overline{F^n(V\setminus D)}\setminus D\right)} \\
& = &  \overline{F(A_n\setminus D)}.
\end{eqnarray*}
It remains to show that $\overline{F(A_n\setminus D)}\subseteq
A_{n+1}$. Since $A_{n+1}$ is closed (in $\mathbb{R}^2$), it suffices
to show that $F(A_n\setminus D)\subseteq A_{n+1}$. Since
$F(A_n\setminus D)=F(A_n\cap (V\setminus D))$, $A_n\cap(V\setminus
D)$ is closed in $V\setminus D$, and $F$ is continuous on
$V\setminus D$, it then follows that
\begin{eqnarray*}
F(A_n\setminus D) & = & F(A_n\cap (V\setminus D)) \\
& = & F\left(\overline{F^n(V\setminus D)}\cap (V\setminus D)\right) \\
& \subseteq & \overline{F(F^n(V\setminus D)\cap (V\setminus D))} \\ 
& = & A_{n+1}. \\ 
\end{eqnarray*}
Therefore $A_{n+1}=\overline{F(A_n\setminus D)}$.

We also note that, following (F3), $f: I\setminus \{ 0\}\to
I\setminus \{ r^+, r^-\}$ is a surjective map. Combining this fact
with (F4) and (F5) we conclude that
\begin{equation} \label{A-n-rho}
\rho^-, \, \rho^+\in A_n, \quad n\geq 0.
\end{equation}


It now  follows from
(\ref{A-n-description}) and (\ref{A-n-rho}) that
\begin{equation} \label{A-n-V}
A_n=\overline{F(A_{n-1}\setminus D)}=F^+(A_{n-1}\cap V^+)\bigcup
F^-(A_{n-1}\cap V^-).
\end{equation}
Let $\mathcal{C}=\{ K\subseteq V \, | \, \mbox{$K$ is closed in
$V$}\}$ and let $\alpha$ be a map from $\mathcal{C}$ to
$\mathcal{C}$ defined by the following formula: for each
$K\in\mathcal{C}$,
\begin{equation} \label{alpha-1}
\alpha (K)=F^+(K\cap V^+)\bigcup F^-(K\cap V^-)\bigcup \{ \rho^-,
\rho^+\}.
\end{equation}
Note that $F(K\cap (V\setminus D))=F^+(K\cap (V^+\setminus
D))\bigcup F^-(K\cap (V^-\setminus D))$. Then it follows from (F5)
that $\overline{F(K\cap (V\setminus D))}\bigcup \{ \rho^-, \rho^+\}
= F^+(K\cap V^+)\bigcup F^-(K\cap V^-)\bigcup \{ \rho^-, \rho^+\}$;
i.e.,
\begin{equation} \label{alpha-2}
\alpha (K)= \overline{F(K\cap (V\setminus D))}\bigcup \{ \rho^-,
\rho^+\}.
\end{equation}
In order to prove that $\{ A_n\}$ is computable, it suffices to show
that the map $\alpha$ is computable since $A_n = \alpha^{n}(V)$.
Towards this end, it suffices to show that, for each $K\in
\mathcal{C}$, (i) $\alpha$ maps every outer-name of $K$ to an
outer-name of $\alpha(K)$ and (ii) $\alpha$ maps every inner-name of
$K$ to an inner-name of $\alpha(K)$. For (i): for each $K\in
\mathcal{C}$, it is clear that $K$ is contained in the ball centered
at the origin of $\mathbb{R}^2$ with a computable radius $\max\{
r^+, 27\}$. Then it follows from Theorems 5.1.13(2) and 6.2.4(4)
\cite{Wei00} that $\alpha$ meets the condition (i). For (ii): Since
$V\setminus D$ is r.e. open in $V$, there exist computable sequences
$\{ a_n\}$ and $\{ s_n\}$, $a_n\in V\setminus D$ with rational
coordinates and $s_n\in \mathbb{Q}$, such that $V\setminus D =
(\bigcup_{n\in\mathbb{N}}B(a_n, s_n))\bigcap V$, where $B(a, s)$ is
an open ball centered at $a$ with radius $s$. Now let $\{ d_j\}$ be
an inner-name of $K$; i.e., $\{ d_j\}$ is a sequence dense in $K$.
For each $m\in\mathbb{N}$, compute the Euclidean distance $d(d_j,
a_n)$ for all $1\leq j, n\leq m$. This algorithm yields a sequence
$\{ d_{j_i}\}$ that is a subset of $\{ d_j\}$, where $x\in
\{d_{j_i}\}$ if and only if $d(x, a_n)<s_n$ for some
$n\in\mathbb{N}$. It is readily seen that $\{ d_{j_i}\}$ is a dense
sequence in $K\cap (V\setminus D)$. Since $F$ is continuous on
$V\setminus D$, it follows that $F(\{ d_{j_i}\})$ is a dense
sequence in $F(K\cap (V\setminus D))$; thus $F(\{ d_{j_i}\})\cup \{
\rho^-, \rho^+\}$ is an inner-name of $\overline{F(K\cap (V\setminus
D))}\bigcup \{ \rho^-, \rho^+\}$, which equals $\phi(K)$ by
(\ref{alpha-2}).
\end{proof}

\begin{proposition}
\label{claim-3} The distance function $d_{A}$ is computable from $F$
and $\rho^{\pm}$.
\end{proposition}

\begin{proof}  It suffices to show that $d_{A_{n}}$
meet the convergence condition of Lemma \ref{Lm:Weihrauch}-(d). The
proof makes use of the properties (F-3)
and (F-4).
Note that it follows from (F-3) that
$f^{n}([r^{-},0)\cup(0,r^{+}])=(r^{-},r^{+})$ for each positive
integer $n$; and from (F-4) that the distance between
$F^{n}(x,y_{1})$ and $F^{n}(x,y_{2})$ decreases exponentially in
$n$:
\[
d(F^{n}(x,y_{1}),F^{n}(x,y_{2}))<c^{n}|y_{1}-y_{2}|
\]
We also observe from (\ref{A-n-description}) that $A_{n+1}\subset
A_n$, $n\in\mathbb{N}$. In the following we show that, for any
$n\in\mathbb{N}$, $F^{n}(V\setminus D)$ is contained in a
$108c^{n}$-neighborhood of $F^{n+1}(V\setminus D)$; thus the
Hausdorff distance between $A_{n}$ and
$A_{n+1}$ is bounded by $108c^{n}$ (recall that $A_{n}=\overline{F^{n}%
(V\setminus D)}$). This fact shows that $d_{A_{n}}$ indeed meet the
convergence condition desired. For any $s\in F^{n}(V\setminus D)$,
there exists $(x,y)\in V$ such that $s=F^{n}(x,y)$. If $x\neq r^{-}$
and $x\neq r^{+}$, then it follows from the fact $f(I\setminus \{
0\})=(r^-, r^+)$ that
there exist $(u,v)\in V$ and $-27\leq w\leq 27$ such that
$F(u,v)=(x,w)$ and $(x, w)$ is in the domain of $F^n$; subsequently,
\begin{align*}
d(F^{n+1}(u,v), s) &  =d(F^{n}(F(u,v)),F^{n}(x,y))\\
&  =d(F^{n}(x,w),F^{n}(x,y))\\
&  \leq c^{n}|w-y|\leq54c^{n}%
\end{align*}
(note that $|w-y|\leq 54$). The above inequality shows that $s$ is
in a $54c^{n}$-neighborhood of $t$, where $t=F^{n+1}(u,v)\in
F^{n+1}(V\setminus D)$.
Next we consider the case where $s=F^n(x, y)$ and $x=r^{-}$ (thus
$y=t^-$). Since $f^n(I\setminus \{ 0\})=(r^-, r^+)$,
there exists $(\tilde{x},\tilde{y})\in V$ such that
$\tilde{x}\neq0,r^{-}$ nor $r^{+}$, and $d(F^{n}(x,y),F^{n}(\tilde{x}%
,\tilde{y}))\leq 54c^{n}$. We now apply the above argument to $\tilde{s}%
=F^{n}(\tilde{x},\tilde{y})$ to find $(u,v)\in V$ and $-27\leq w\leq
27$ such that $F(u,v)=(\tilde{x},w)$. It then follows that
\begin{align*}
&  d(F^{n+1}(u,v),F^{n}(x,y))\\
&  \leq d(F^{n+1}(u,v),F^{n}(\tilde{x},\tilde{y}))+d(F^{n}(\tilde
{x},\tilde{y}),F^{n}(x,y))\\
&  \leq54c^{n}+54c^{n}=108c^{n}%
\end{align*}
in other words, $s$ is in the $108c^{n}$-neighborhood of
$t=F^{n+1}(u, v)\in F^{n+1}(V\setminus D)$. The same argument
applies to the case where $x=r^{+}$. Thus we have shown that for any
$s\in F^{n}(V\setminus D)$ there exists $t\in F^{n+1}(V\setminus D)$
such that $s$ is in the $108c^{n}$-neighborhood of $t$. Hence the
Hausdorff distance between $A_{n}$ and $A_{n+1}$ is bounded by
$108c^{n}$.
\end{proof}

\medskip

Before proving our main result, we need one more Lemma, that will prove also useful in the next section.

\begin{lemma}\label{poincaremap}
Let $\phi$ be the flow of some Lorenz geometric system. Then we can uniformly compute from a ($C^2$) name of $\phi$:
\begin{enumerate}
\item  The return function $F$ (and its components $f,g$).
\item The return time function
$r:V\backslash D\rightarrow\lbrack0,+\infty)$.
\item The points $r^{\pm},t^{\pm}$.
\end{enumerate}
\end{lemma}

\begin{proof} Because we have access to a $C^2$ name of $\phi$, we can compute its derivative and
hence we can compute the function $h:D\subseteq\mathbb{R}^{3}\rightarrow
\mathbb{R}^{3}$ defining the ODE
\begin{equation}
y^{\prime}=h(y), \quad \label{ODE:Lorenz-1}%
\end{equation}
whose flow is $\phi$. Let us now show the condition (2) of the
Lemma.

The idea for the proof is relatively simple, that is, computing the
time that a trajectory starting in a point $a\in V\times\{27\}$
needs to hit $V\times\{27\}$ again. The strategy is to compute
iterates $\phi_{t_i}(a)$ for $i=1,2,\ldots$ until the iterate is on
(or close enough to) $V\times\{27\}$. The difficulty is that we need
to be careful on the way how we choose the time step needed to
 compute the time $t^*>0$ when $\phi_{t}(a)$ hits $V\times\{27\}$
for the first time,
to avoid returning some $t^{**}>t^*$ with $\phi_{t^{**}}\in V\times\{27\}$.


Since the flow of (\ref{ODE:Lorenz-1}) behaves like the geometric
Lorenz attractor, we conclude that the flow will cross the
cross-section $V$ transversally, which has the direction of the
positive $z$-axis as its ``normal" direction. This implies that for
any point $a\in V$, the angle $\measuredangle(h(a),V)$ between
$h(a)$ and the cross-section $V\times \{27\}$ will satisfy
$\measuredangle(h(a),V)\neq0$.
Let $\theta=\frac{\min_{a\in V}\left\vert
\measuredangle(h(a),V)\right\vert }{2}>0$. Then there exists some
$\varepsilon>0$ such that%
\[
\min_{a\in V\times\lbrack27-\varepsilon,27+\varepsilon]}\left\vert
\measuredangle(h(a),V)\right\vert >\theta>0\text{.}%
\]
(Recall that $V$ is compact and thus the minimum exists.) Initially
the flow on $V\times \{27\}$ will be pointing downwards, i.\ e.\
$\measuredangle (h(a),V)<0$. Let $a\in V\backslash D$ and suppose
that we want to compute $r(a)$ with precision bounded by some value
$\epsilon>0$, i.\ e.\ we want to
compute a value $\tilde{r}_{a}$ such that%
\begin{equation}
\left\vert r(a)-\tilde{r}_{a}\right\vert \leq\epsilon
\label{Eq:ReturnTimeApprox}%
\end{equation}
To prove this result we will use an \textquotedblleft
adapted\textquotedblright\ Euler method to compute $\tilde{r}_{a}$.
The idea is to numerically compute the solution of
(\ref{ODE:Lorenz-1}) starting at $a$ using an algorithm which
discretizes time steps, similar to Euler's method. However the time
steps must be chosen small enough so that we can detect when the
flow first leaves the band
$V\times\lbrack27-\varepsilon,27+\varepsilon]$, and then when it
re-enters this band again (from the top). In this manner, by
improving the accuracy of the numerical method and/or using a
smaller $\varepsilon$, we will be able to compute a suitable
approximation $\tilde{r}_{a}$ for the return time $r(a)$ which
satisfies condition (\ref{Eq:ReturnTimeApprox}). Of course, we have
to describe more precisely how this method works.

Let
\[
\alpha=\min_{a\in V\times\lbrack27-\varepsilon,27+\varepsilon]}\left\Vert
h(a)\right\Vert ,\text{ \ \ }\beta=\max_{a\in V\times\lbrack27-\varepsilon
,27+\varepsilon]}\left\Vert h(a)\right\Vert
\]
Note that $\alpha,\beta>0$ since $V$ is compact and contains no
zeros of $h$. A simple analysis (consider the component of the flow
$h_{V}(b)$ which is orthogonal to $V\times \{27\}$, given by
$h_{V}(b)=\left\Vert h(b)\right\Vert \left\vert
\sin(\measuredangle(h(b),V))\right\vert $, for any $b\in V\times
\lbrack27-\varepsilon,27+\varepsilon]$, which satisfies
$\alpha\sin\theta \leq\left\Vert h(b)\right\Vert \sin\theta\leq
h_{V}(b)\leq\left\Vert h(b)\right\Vert \leq\beta$) shows that the
flow of (\ref{ODE:Lorenz-1}) cannot take more than
$2\varepsilon/(\alpha\sin\theta)>0$ time units to cross the band
$V\times\lbrack27-\varepsilon,27+\varepsilon]$ (basically the flow
will have to cross this band; but since the norm of the orthogonal
component is at least $\alpha\sin\theta$, this will be done in time
$2\varepsilon/(\alpha \sin\theta)$), but will require at least
$2\varepsilon/\beta>0$ time units to cross it (since the norm of the
orthogonal component is bounded by $\beta$). Now pick some rational
$\varepsilon_{0}>0$ satisfying $\varepsilon_{0}\leq
\min\{\epsilon\alpha\sin\theta/2,\varepsilon\}$. In particular this
implies that the maximum time the flow takes to cross the band
$B=V\times\lbrack 27-\varepsilon_{0},27+\varepsilon_{0}]\subseteq
V\times\lbrack27-\varepsilon
,27+\varepsilon]$ is%
\begin{equation}
2\varepsilon_{0}/(\alpha\sin\theta)\leq\frac{2\epsilon\sin\theta}{2\sin\theta
}\leq\epsilon\label{aux1}%
\end{equation}
This implies that if we can tell that the flow starting at $a$ leaves and then
re-enters the band $V\times\lbrack27-\varepsilon_{0},27+\varepsilon_{0}]$ from
the top for the first time at time $T_{0}$, and stays in this band up to time
$T_{0}^{\ast}$, with $T_{0}<T_{0}^{\ast}$ (note that $T_{0}^{\ast}%
-T_{0}<\epsilon$ due to (\ref{aux1})) and if we can determine a time
$T\in\lbrack T_{0},T_{0}^{\ast}]$, then we can return $\tilde{r}_{a}=T$ since
condition (\ref{Eq:ReturnTimeApprox}) holds in that case. Now let us see how
we can determine this value $T$.

Let $B^+=V\times\lbrack 27,27+\varepsilon_{0}]$,
$B^-=V\times\lbrack 27-\varepsilon_{0},27]$, and $\delta=\varepsilon_{0}/(2\beta)$.
Now consider the sequence of iterates $\phi_{t_i}(a)$ where $0<t_{i+1}-t_i\leq\delta$ and
$\{t_i\}_{i\in\N}$ is computable.
Since the flow of $\phi_t(a)$ takes at least $2\delta$ time units to cross each
band $B^\pm$, we are certain that $\phi_{t_1}(a),\phi_{t_2}(a)\in B^-$
when the flow first leaves $V$ from $a$ and that there is some $k>0$ such that
$\phi_{t_k}(a),\phi_{t_{k+1}}(a)\in B^+$ with $t_k,t_{k+1}\in \lbrack T_0,T_{0}^{\ast}]$.

Note that the interior of $B^+$ is a r.e. open sets, as well as its complement. Since at
every time $t_i$ the corresponding iterate is computable, and because one can
semi-decide whether a computable point belongs to a r.e. open set, we can semi-decide
in parallel, for each $i\in\N$, whether the iterates $\phi_{t_i}$ or $\phi_{t_{i+1}}$ belong to
$B^+$ or to its complement. Since only one of these iterates can fall exactly in
the boundary of $B^+$ (which is the only thing one cannot detect), we know that we
can tell in finite time, for at least one of the iterates, whether it belongs to
$B^+$ or to its complement. Now run this procedure as a subroutine
for each pair $t_i$, $t_{i+1}$. Start with $i=1$ and increment $i$ each time we conclude that an
iterate $\phi_{t_j}$, for $j\in\{i,i+1\}$ does not belong to $B^+$. If we conclude that
some iterate $\phi_{t_j}$ belongs to $B^+$ then stop the algorithm and return $\tilde{r}_{a}=t_j$.

Note that this algorithm always stops, in the worst case, when $i=k$, and therefore always computes
the return time.

To prove condition (1) of the Lemma, we note that $F(a)$ is the solution
of \eqref{ODE:Lorenz-1} with initial condition $y^{\prime}(0)=a$ at time $r(a)$.
Since $r$ is computable from $\phi$ and the solution of \eqref{ODE:Lorenz-1} is
also computable from $h,a$ \cite{GZB07} and hence from $\phi,a$, we conclude that $F$
is computable from $\phi$.

To prove condition (3) of the Lemma, we notice that the stable manifold of the
origin is locally computable from $h$ \cite{GZB12}. If we compute a local version
of the stable manifold which stays on the half-space $z<27$, and if we take some point
from that local stable manifold which is not the origin, we know that the trajectory
starting from this point will move upwards, until it reaches the plane $z=27$ and then
continues moving up, until it falls and reaches the plane $z=27$ for the second time.
At this time the intersection will occur at $\rho^-$ or $\rho^+$, depending on whether
the first coordinate of this intersection point is  positive or negative, respectively.
Hence, using similar arguments as those used for the cases (2) and (1), we conclude
that $\rho^{\pm}$ must be computable and hence $r^{\pm}, t^{\pm}$ are also computable from $\phi$.
\end{proof}

We are now in position to prove our first main result.

\begin{theorem}
\label{Th:GeometricAttractor} The global attractor $\mathcal{A}$ of a
geometric Lorenz flow $\phi$ is computable from a $(C^2)$ name of $\phi$.

\end{theorem}

\begin{proof} By Lemma \ref{poincaremap}, we only need to show that the operation $(\phi, F, r^{\pm})\to\mathcal{A}$
is computable. To prove that $\mathcal{A}$ is computable from $\phi, F$, and
$r^{\pm}$, it suffices to show that, from the given information, (i)
a sequence dense in $\mathcal{A}$ can be computed and (ii) a
sequence of open rational balls exhausting the complement of
$\mathcal{A}$ can be computed.

For $x\in V$ and $T>0$, let $O_{T}(x)=\{ \phi(t,x) : -T \leq t \leq T \}$ and
$O_{T}(A)=\cup_{x\in A}O_{T}(x)$. Then $\mathcal{A}=O_{\infty}(A)\bigcup\{ (0,
0, 0)\}$. Since for each positive rational number $T$, the compact subset
$O_{T}(A)$ of $\mathbb{R}^{3}$ is computable from $\phi, T$, and $A$ by Lemma
\ref{Lm:Weihrauch}-(a), it follows from Theorem \ref{Th:ReturnMapAttractor} that a sequence dense in $O_{T}(A)$ can be
computed using the given information. By effectively listing the set of all
positive rational numbers and then using a computable pairing function, we
obtain a sequence dense in $O_{\infty}(A)$, which is of course also dense in
$\mathcal{A}$. This proves (i).

We now turn to (ii). It is enough to show that given a point $x \in M$
we can semi-decide, uniformly in $x$, whether $x$ is outside the global
attractor $\mathcal{A}$, that is, whether $x\notin \mathcal{A}$.
By the proof of Lemma \ref{poincaremap}, we know that we can use $\phi$
to follow the trajectory starting at $x$ until it hits $V$ for the first time,
and then compute the point $l(x) \in V$, at which this trajectory lands. Note
that $l(x)=\phi_t(x)$  for some (computable) time $t$.  It follows that $x \in \mathcal{A}$
if and only if $l(x)\in A$, and this last relation can be semi-decided by Theorem \ref{Th:ReturnMapAttractor}. This proves (ii).

\end{proof}

\begin{corollary}
The geometric Lorenz attractor contains computable points with dense orbits.
\end{corollary}

\begin{proof}
By the previous result, $A$ itself is a computable metric space. The
Poincar\'e map on $A$ is well defined and computable on $A\setminus D$ which,
with respect to the induced topology on $A$ is a recursively enumerable open set which is dense on $A$. Moreover, this
dynamical system is transitive (see for instance \cite{GH83}) and therefore it
contains a computable point whose orbit is dense in $A$ (see \cite{GHR09b}, Theorem 3).
But the orbit of this point under the flow is dense in
$\mathcal{A}$, which finishes the proof.
\end{proof}

\section{A computable geometric Lorenz flow admits a computable physical
measure}

\label{Sec:SBRmeasure}

Given an invariant probability measure $\mu$ for a flow $\phi_{t} $ on a space
$M$, let $\mathcal{B}(\mu)$ be the set of initial conditions $z\in M$
satisfying for all continuous functions $\varphi:M\to\mathbb{R}$:
\[
\lim_{T\to\infty}\frac{1}{T}\int_{0}^{T}\varphi(\phi_{t}(z))\,dt = \int_{M}
\varphi(z)\, d\mu.
\]
The set $\mathcal{B}(\mu)$ is known as the (ergodic) \emph{basin} of $\mu$.
When this basin has positive volume, one says that the measure $\mu$ is
\emph{Physical}, or $\emph{SRB}$ (for Sinai-Ruelle-Bowen, see for instance
\cite{young2002srb}). These measures are ``physical" in the sense that they
describe the statistical asymptotic behavior for a ``big" (positive volume)
set of initial conditions, so they represent the ``physically observable"
equilibrium states of the system.

Geometric Lorenz attractors are robust attractors of 3-dimensional flows, and
it was shown in \cite{APPV09} that they admit a unique physical measure. In this section, we show that
if the data defining a geometric Lorenz flow are computable, then the flow
admits a computable physical measure.

We start by recalling the definition of computable measure.

\begin{definition}
A probability measure $\mu$ on a (computably) compact subset $M\subset\mathbb{R}^{3}$
is computable if the integration operator $\varphi\to\int_{M} \varphi\,d\mu$,
where $\varphi$ is a continuous real valued function on $M$, is computable.
\end{definition}

It can be shown (see for instance \cite{Roj08}) that if $R:M\to
M^{\prime}$ is a computable function and $\mu$ is a computable probability
measure on $M$, then the push forward $R^{*}\mu$ of $\mu$ by $R$, defined by
\[
R^{*}\mu(E)=\mu(f^{-1}(E))
\]
is also a computable measure.

\begin{theorem}
\label{Th:PhysicalMeasure}Let $\phi$ be the flow of some Lorenz geometric system.
If $\phi$ is ($C^2$) computable, then the geometric Lorenz flow admits a computable
physical measure. More generally, the geometric Lorenz flow admits a physical measure which is computable from $\phi$.
\end{theorem}

\begin{proof}
Let $F: V^{-}\bigcup V^{+}\to V$, $F(x, y)=(f(x), g(x, y))$, be the
return map of the geometric Lorenz flow, as defined in Subsection
\ref{Subsec:LorenzGeometric}. The map $f: I\setminus\{0\}\to I$, $I=[r^{-},
r^{+}]$, describes the dynamics of the leaves $\{\gamma_{x}\}_{x\in I}$ of the
foliation $\mathcal{F}$ of $V$, which is invariant for the return map $F$
(recall that the leaves are just vertical straight lines $x=c$). In
particular, for each $x\in I$ and $x\neq0$,
\[
F(\gamma_{x})\subset\gamma_{f(x)}.
\]
Moreover, the dynamics of $F$ is uniformly contracting in the direction of the
leaves of $\mathcal{F}$.

Since $f$ is expanding, it follows that it admits a unique ergodic invariant
measure $\mu_{f}$ on $[r^{-},r^{+}]$ which is absolutely continuous with
respect to Lebesgue measure (see for instance \cite{viana1997stochastic}).  Moreover, it can be shown that
this measure has a bounded density function. Recall that by Lemma \ref{poincaremap},
the functions $F$ and $f$ are computable from $\phi$.  It follows from \cite{GHR12} that
$\mu_{f}$ is also computable from $\phi$.

One then considers the product measure $\nu=\mu_{f}\times\mu_{L}$ on $V$,
where $\mu_{L}$ is just the Lebesgue measure on $[-27,27]$, normalized to
integrate one. It is easy to see that $\nu$ is a computable measure too. By
the contracting property of $F$ on the leaves, it follows that the
push-forwards $F^{*}\nu$ of this measure by $F$, defined by
\[
F^{*n}\nu(E) = \nu(F^{-n}E),
\]
converge exponentially fast (in the weak* topology) towards a limit measure
$\mu_{F}$ on $V$ which is invariant and physical for $F$ (see \cite{araujo2010three}). The
sequence $F^{*n}\nu$ being computable, as well as the rate of convergence,
imply computability of the limit measure $\mu_{F}$.

The last step is to compute a physical measure for the flow. To this end, let
$V^{*r}$ be the subset of $\mathbb{R}^{3}$ defined by%

\[
V^{*r} = \{ (x,y,z)\in\mathbb{R}^{3} : (x,y) \in V\setminus D, z\in[0,r(x,y)]\}.
\]

In case the function $r$ is integrable,
\[
\int_{V\setminus D} r(x,y) d\mu_{F} < \infty,
\]
a measure $\mu^{*}$ on $V^{*r}$ can be naturally defined by:%

\[
\mu^{*}= \frac{ \mu_{F} \times\mu_{L} }{\int r(x,y) d\mu_{F}}
\]
where $\mu_{L}$ is again Lebesgue measure. Moreover, this measure is
computable whenever the integral above is computable. We then transport this
measure into the actual flow via the function
\[
\Phi: V^{*r} \to M : \newline \quad(x,y,t) \to\phi_t(x,y,27)
\]
where $\phi_t(x,y,z)$ is the trajectory of the flow at time $t$ starting at
$(x,y,z)$.
Clearly, the function $\Phi$ is  computable from $\phi$, which implies that the
transported measure:%

\[
\mu_{Physical} (E) = \mu^{*} (\Phi^{-1} E)
\]
where $E$ is a Borel set of $M$, is a computable measure.
Moreover, by \cite{araujo2010three}, this is the physical measure for the flow. The
following claim therefore finishes the proof of the Theorem.
\end{proof}

\begin{claim}
$\displaystyle{\int_{V^{*}} r(x,y) d\mu_{F}}$ is computable.
\end{claim}

\begin{proof}[Proof of the Claim] Since the return function $r(x,y)$ depends only on the $x$
coordinate, we have that $\int_{V^{*}} r(x,y) d\mu_{F} = \int_{I} r(x)
d\mu_{f} $, where $r(x)$ is the projection of $r$ onto $I$. We have already
seen that $r(x)$ is a computable unbounded function on $I\setminus\{0\}$
(Lemma \ref{poincaremap}).
The following estimate is shown in \cite{luzzatto2005lorenz}:
\[
|r(x)-r(y)|\leq C|\ln|x|-\ln|y| |
\]
for all $x,y>0$ and all $x,y<0$, where $C\geq1$ is a constant. We show that
$\int_{(0,1]}r(x)d\mu_{f}$ is computable. Since $r(x)$ is computable and
bounded on $[\epsilon,1]$, we have that $\int_{\epsilon}^{1} r(x) d\mu_{f}$ is
computable. Thus, we only need to estimate $\int_{0}^{\epsilon}r(x) d\mu_{f}$.
By the inequality above, we have that $|r(x)-r(1)|\leq C |\ln|x|-\ln|1||=C
|\ln|x||$ so that, for $x>0$ we have $r(x) \leq C|\ln(x)| + r(1)$. Recall that
$\mu_{f}$ is absolutely continuous with density bounded above, say by $M$.
Then
\[
\int_{0}^{\epsilon}r(x) d \mu_{f} \leq M \int_{0}^{\epsilon}r(x) \,d x \leq
M(C \epsilon[\ln(1/\epsilon) +1] + \epsilon r(1)) =O\big(\epsilon
\ln(1/\epsilon)\big).
\]
The claim then follows.
\end{proof}



\bibliographystyle{plain}
\bibliography{CompLorenz}

\begin{thebibliography}{10}

\bibitem{ABS77}
V.~S. Afraimovich, V.~V. Bykov, and L.~P. Shil'nikov.
\newblock On the appearence and structure of the lorenz attractor.
\newblock {\em Dokl. Acad. Sci. USSR}, 234:336--339, 1977.

\bibitem{APPV09}
V.~Araujo, M.J. Pacifico, R.~Pujals, and M.~Viana.
\newblock Singular-hyperbolic attractors are chaotic.
\newblock {\em Trans. Amer. Math. Soc.}, 361:2431--2485, 2009.

\bibitem{araujo2010three}
V{\'\i}tor Ara{\'u}jo and Maria~Jos{\'e} Pacifico.
\newblock {\em Three-dimensional flows}, volume~53.
\newblock Springer Science \& Business Media, 2010.

\bibitem{BHW08}
V.~Brattka, P.~Hertling, and K.~Weihrauch.
\newblock A tutorial on computable analysis.
\newblock In S.~B. Cooper, , B.~L{\"{o}}we, and A.~Sorbi, editors, {\em New
  Computational Paradigms: Changing Conceptions of What is Computable}, pages
  425--491. Springer, 2008.

\bibitem{BY06}
M.~Braverman and M.~Yampolsky.
\newblock Non-computable {J}ulia sets.
\newblock {\em J. Amer. Math. Soc.}, 19(3):551--578, 2006.

\bibitem{GHR12}
S.~Galatolo, M.~Hoyrup, and C.~Rojas.
\newblock Statistical properties of dynamical systems -- simulation and
  abstract computation.
\newblock {\em Chaos, Solitons \& Fractals}, 45:1--14, 2012.

\bibitem{GHR09b}
Stefano Galatolo, Mathieu Hoyrup, and Crist{\'o}bal Rojas.
\newblock A constructive borel--cantelli lemma. constructing orbits with
  required statistical properties.
\newblock {\em Theoretical Computer Science}, 410(21):2207--2222, 2009.

\bibitem{GZB07}
D.S. Gra{\c{c}}a, N.~Zhong, and J.~Buescu.
\newblock Computability, noncomputability and undecidability of maximal
  intervals of {IVP}s.
\newblock {\em Trans. Amer. Math. Soc.}, 361(6):2913--2927, 2009.

\bibitem{GZB12}
D.S. Gra{\c{c}}a, N.~Zhong, and J.~Buescu.
\newblock Computability, noncomputability, and hyperbolic systems.
\newblock {\em Appl. Math. Comput.}, 219(6):3039--3054, 2012.

\bibitem{GH83}
J.~Guckenheimer and P.~Holmes.
\newblock {\em Nonlinear Oscillations, Dynamical Systems, and Bifurcation of
  Vector Fields}.
\newblock Springer, 1983.

\bibitem{GW79}
J.~Guckenheimer and R.~F. Williams.
\newblock Structural stability of lorenz attractors.
\newblock {\em Publ. Math. IHES}, 50:59--72, 1979.

\bibitem{HSD04}
M.~W. Hirsch, S.~Smale, and R.~Devaney.
\newblock {\em Differential Equations, Dynamical Systems, and an Introduction
  to Chaos}.
\newblock Academic Press, 2004.

\bibitem{Lor63}
E.~N. Lorenz.
\newblock Deterministic non-periodic flow.
\newblock {\em J. Atmos. Sci.}, 20:130--141, 1963.

\bibitem{luzzatto2005lorenz}
Stefano Luzzatto, Ian Melbourne, and Frederic Paccaut.
\newblock The lorenz attractor is mixing.
\newblock {\em Communications in Mathematical Physics}, 260(2):393--401, 2005.

\bibitem{Palis}
J.~Palis.
\newblock A global view of dynamics and a conjecture on the denseness of
  finitude of attractors.
\newblock {\em Ast{\'e}risque}, 261:339 -- 351, 2000.

\bibitem{Roj08}
C.~Rojas.
\newblock {\em Randomness and Ergodic Theory: an Algorithmic point of view}.
\newblock PhD thesis, \'{E}cole Polytechnique \& Universit\`{a} di Pisa, 2008.

\bibitem{Sma98}
S.~Smale.
\newblock Mathematical problems for the next century.
\newblock {\em Math. Intelligencer}, 20:7--15, 1998.

\bibitem{Tuc02}
W.~Tucker.
\newblock A rigorous ode solver and smale's 14th problem.
\newblock {\em Found. Comput. Math.}, 2(1):53--117, 2002.

\bibitem{viana1997stochastic}
Marcelo Viana.
\newblock {\em Stochastic dynamics of deterministic systems}, volume~21.
\newblock IMPA Rio de Janeiro, 1997.

\bibitem{Wei00}
K.~Weihrauch.
\newblock {\em Computable Analysis: an Introduction}.
\newblock Springer, 2000.

\bibitem{young2002srb}
Lai-Sang Young.
\newblock What are srb measures, and which dynamical systems have them?
\newblock {\em Journal of Statistical Physics}, 108(5):733--754, 2002.

\bibitem{ZW03}
N.~Zhong and K.~Weihrauch.
\newblock Computability theory of generalized functions.
\newblock {\em J. ACM}, 50(4):469--505, 2003.

\end{thebibliography}

\end{document}